\pdfoutput=1
\documentclass[11pt]{amsart}
\usepackage[paper,pkg={amsaddr=true,tikz=true,biblatex/opt={+useprefix=true}}]{zbs}
\DeclareMathOperator{\ex}{ex}
\newcommand{\calp}{\mathcal{P}}
\newcommand{\cals}{\mathcal{S}}
\newcommand{\cG}{\mathcal{G}}
\newcommand{\sM}{\mathsf{M}}
\newcommand{\astk}{\mathord{*}}
\newcommand{\bslash}{\mathbin{/}}

\tikzset{vtx/.style={inner sep=1.5pt,fill,circle}}

\title{Series-parallel graphs and hypercube degeneracy}
\author{Daniel G.\ Zhu}
\email{zhd@princeton.edu}
\address{Department of Mathematics, Princeton University, Princeton, NJ 08544, USA}

\begin{document}
\begin{abstract}
Several recent works have identified patterns that must exist in dense subsets of either the vertices or the edges of a large hypercube. We introduce a framework, based on the concept of series-parallel graphs, that unifies and generalizes these results.
\end{abstract}
\maketitle

\section{Introduction}
In extremal combinatorics, one of the most fundamental questions one can ask is a Tur\'an-type question, which, broadly speaking, asks to determine $\ex(\cals, \calp)$, the maximum size of a subset of some large ``host structure'' $\cals$ that does not contain some pattern $\calp$. For instance, a classical question concerns the case when $\cals$ is (the edges of) the complete graph $K_n$ on $n$ vertices, and $\calp$ is some graph $H$. In this case, there is a fundamental dichotomy for the behavior of $\ex(K_n, H)$: if $H$ is bipartite, the Tur\'an problem is said to be ``degenerate'' and $\ex(K_n, H) = o(n^2)$, whereas otherwise, we have $\ex(K_n, H) = \Theta(n^2)$ (see e.g.\ \cite{FSdegen}). Similar results hold in other settings as well: in the case of $r$-uniform hypergraphs the degenerate patterns are the $r$-partite $r$-uniform hypergraphs, while in the case of directed graphs the degenerate patterns are the (bipartite) directed graphs where every vertex is a source or a sink.

The topic of this paper is the classification of degenerate patterns in two related settings: vertices and edges of the $n$-dimensional hypercube $Q_n$. Specifically, given nonnegative integers $d \leq n$, define an \vocab{embedding} $Q_d \to Q_n$ to be the identification of $Q_d$ with an $d$-dimensional face of $Q_n$. Then, for a subset $X \subseteq V(Q_d)$, let $\ex(Q_n, X)$ be the maximum size of a subset $S \subseteq V(Q_n)$ such that $\psi(X) \nsubseteq S$ for any embedding $\psi\colon Q_d \to Q_n$. Similarly, for a subset $Y \subseteq E(Q_d)$, let $\ex(Q_n, Y)$ be the maximum size of a subset $S \subseteq E(Q_n)$ such that $\psi(Y) \nsubseteq S$ for any embedding $\psi\colon Q_d \to Q_n$. Call a set of vertices $X$ or a set of edges $Y$ \emph{hypercube-degenerate} (abbreviated to \emph{$h$-degenerate} in this paper) if $\ex(Q_n, X) = o(2^n)$ or $\ex(Q_n, Y) = o(e(Q_n)) = o(n2^{n})$, respectively.

In the case of vertices, this notion of $\ex(Q_n, X)$ was introduced by Johnson and Talbot \cite{JT10}, who made a conjecture equivalent to the statement that the $h$-degenerate $X$ are precisely those that are a subset of some layer of a hypercube. This conjecture was also made independently by Bollob\'as, Leader, and Malvenuto \cite{BLM11}. However, this conjecture was recently disproved by Ellis, Ivan, and Leader \cite{EIL24}, who exhibited a construction, based on linear algebra over the finite field $\setf_2$, which both has positive density and avoids the middle layer of the cube $Q_4$. Expanding on this idea, recent work of Alon \cite{Alon24} tightened both lower and upper bounds for the size of the largest $h$-degenerate subset of $V(Q_d)$.

The case of edges has a longer history and has focused on the weaker notion of graph containment; that is, given a graph $H$, the focus of study has been the quantity $\ex(Q_n, H)$, the largest number of edges of a subgraph of $Q_n$ containing no copy of $H$, irrespective of how $H$ is embedded within the cube. If $\ex(Q_n, H) = o(e(Q_n))$, we call $H$ \vocab{$g$-degenerate}. (As will be discussed later in \cref{sec:containment}, the concept of $h$-degeneracy subsumes $g$-degeneracy.) In 1984, Erd\H{o}s conjectured that all even cycles $C_{2t}$ with $t \geq 3$ are $g$-degenerate \cite{Erd84}, which was later proved for all $t\notin\set{3,5}$ \cite{Chung92,FO11} but disproved for $t = 3$ \cite{Chung92} and $t = 5$ \cite{GM24+}. The $g$-degeneracy of $C_{2t}$ for $t \in \set{4,6,7,8,\ldots}$ was generalized by Conlon \cite{Conlon10}, who showed that all graphs with a \emph{partite representation} (to be defined later in \cref{def:partite}) are $g$-degenerate, a result that was later applied in \cite{Mar22,AMW24}. Recently, Axenovich \cite{Axen23} further showed that graphs whose blocks all have partite representations are $g$-degenerate.

This paper identifies a family of $h$-degenerate vertex patterns and a family of $h$-degenerate edge patterns, which, surprisingly, are most naturally defined in terms of the notion of series-parallel graphs from structural graph theory. As far as we are aware, this framework unifies every known $h$-degenerate pattern in the literature and uncovers many more that were previously unknown. Furthermore, there are philosophical reasons (see \cref{sec:opt}) to believe that our result could be the complete classification of $h$-degenerate patterns, though we do not make this assertion with much confidence, considering the spotty track record of previous conjectures in this area.

To state our results, we recall that a \emph{series-parallel graph} is a multigraph with no $K_4$ minor. Given a connected series-parallel graph $G$, let $X(G) \subseteq 2^{E(G)}$ be the set of edge sets of spanning trees of $G$; given an ordering of $E(G)$, this can be interpreted as a subset of $V(Q_{e(G)})$. Further suppose an edge $e \in E(G)$ is neither a bridge nor a loop, and let $G\bslash e$ and $G \setminus e$ denote $G$ with $e$ contracted and deleted, respectively. Given an ordering of $E(G) \setminus \set{e}$, we may consider $X(G\bslash e)$ and $X(G\setminus e)$ as subsets of $V(Q_{e(G)-1})$, and we let $Y(G, e) \subseteq E(Q_{e(G)-1})$ be the edges connecting an element of $X(G\bslash e)$ and an element of $X(G\setminus e)$. In other words, $Y(G, e)$ is the set of pairs $\set{S, S \cup \set{e'}}$ such that both $S \cup \set{e}$ and $S \cup \set{e'}$ determine spanning trees of $G$. The main result of this paper states that these constructions are $h$-degenerate, with a power-saving term in the Tur\'an number.
\begin{thm} \label{thm:mainv}
If $G$ is a connected series-parallel graph, there exists some $\eps > 0$ such that $\ex(Q_n, X(G)) = O(n^{-\eps} 2^n)$.
\end{thm}
\begin{thm} \label{thm:maine}
If $G$ is a connected series-parallel graph and $e \in E(G)$ is neither a bridge nor a loop, then there exists some $\eps > 0$ with $\ex(Q_n, Y(G, e)) = O(n^{-\eps} e(Q_n))$.
\end{thm}
In fact we are able to show something somewhat stronger; we defer the precise statements to \cref{thm:fullv,thm:fulle}. In \cref{sec:pj}, we also state a partial reformulation of \cref{thm:maine} that does not use the terminology of series-parallel graphs.

As previously mentioned, \cref{thm:mainv,thm:maine} lead to quick proofs of results of Alon \cite{Alon24}, Conlon \cite{Conlon10}, and Axenovich \cite{Axen23}, which we describe in \cref{sec:cor}. In \cref{sec:quant}, we use \cref{thm:mainv} is to improve lower bounds on the size of the largest $h$-degenerate subset of $Q_d$ and discuss the analogous problem for edges. In \cref{sec:small}, we show that in some small cases, \cref{thm:mainv,thm:maine} are complete, in the sense that every $h$-degenerate pattern is a subset of some $X(G)$ or some $Y(G, e)$. Whether this is true in general is unknown; for concreteness, in \cref{sec:small} we also identify what can be considered to be the simplest patterns whose $h$-degeneracy is unknown.

We conclude this introduction with a simple example. Suppose $G$ is $K_4$ minus an edge, with its edges labeled and ordered as follows:
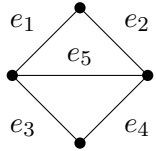
\begin{center}
\begin{tikzpicture}[scale=0.9]
\node[vtx] at (1,0) {};
\node[vtx] at (-1,0) {};
\node[vtx] at (0,1) {};
\node[vtx] at (0,-1) {};
\draw (1,0)-- node[anchor=south west] {$e_2$}(0,1)-- node[anchor=south east] {$e_1$}(-1,0)-- node[anchor=south] {$e_5$}(1,0)-- node[anchor=north west] {$e_4$}(0,-1)--node[anchor=north east] {$e_3$}(-1,0);
\end{tikzpicture}
\end{center}
Then, we have
\begin{align*}
X(G) &= \set{01110,10110,11010,11100,01011,01101,10101,10011} \\
X(G \bslash e_5) &= \set{0101,0110,1001,1010} \\
X(G \setminus e_5) &= \set{0111,1011,1101,1110}.
\end{align*}
Thus $Y(G,e_5) \subseteq E(Q_4)$ is an $8$-cycle given by \[0101\mathord{-}0111\mathord{-}0110\mathord{-}1110\mathord{-}1010\mathord{-}1011\mathord{-}1001\mathord{-}1101\mathord{-}0101.\]
In particular, applying \cref{thm:maine} yields the fact that $C_8$ is $g$-degenerate.

\subsection*{Outline}
In \cref{sec:sp} we list necessary background results on the structure of series-parallel graphs. \cref{sec:main} proves \cref{thm:mainv,thm:maine}, while \cref{sec:add} contains various results related to understanding or applying \cref{thm:mainv,thm:maine}. We conclude with some final remarks in \cref{sec:conc}.

\section{Structural Graph Theory Preliminaries} \label{sec:sp}
In this section, we recall some results and terminology on the theory of series-parallel graphs. In this paper, the term \vocab{graph} on its own refers to a simple graph,\footnote{In general, the term ``graph'' is only used in contexts where it is obvious that multigraphs are not part of the discussion, such as when discussing $g$-degeneracy.} but the term \vocab{series-parallel graph} includes multigraphs.

Given two multigraphs $G_1$ and $G_2$ with distinguished vertices $v_1 \in V(G_1)$ and $v_2 \in V(G_2)$, the \vocab{$1$-sum} of $G_1$ and $G_2$ is the multigraph obtained by joining $G_1$ and $G_2$ by identifying $v_1$ and $v_2$. Given two multigraphs $G_1$ and $G_2$ with distinguished non-loop edges $e_1 \in E(G_1)$ and $e_2 \in E(G_2)$, the \vocab{$2$-sum} of $G_1$ and $G_2$ is the multigraph obtained by joining $G_1$ and $G_2$ by identifying $e_1$ and $e_2$ (under some orientation). Note that some definitions of the $2$-sum in the literature allow or require the distinguished edge of a $2$-sum to be deleted; this is not the case here.

A multigraph is \vocab{$2$-connected} if it has at least two edges, has no loops, and is such that deleting every vertex yields a connected graph. Any multigraph can be constructed from loops, single edges, and $2$-connected graphs by taking disjoint unions and $1$-sums. Moreover, these constituent graphs are recoverable from the final graph and partition its edges into \vocab{blocks}. For example, the blocks that are edges are the bridges of the graph.

Let $\cG$ be the set of connected series-parallel multigraphs. Let $\cG'$ be the set of connected series-parallel multigraphs with a single distinguished edge that is neither a bridge nor a loop. Furthermore, let $\cG''$ be the set of $(G, e) \in \cG'$ such that $G$ is $2$-connected. We will use the following operations on $\cG$, $\cG'$, $\cG''$:
\begin{itemize}
\item \textbf{Loop addition} ($\cG \to \cG$, $\cG' \to \cG'$): adding a loop to a vertex. This loop is not distinguished in the cases of $\cG'$ and $\cG''$.
\item \textbf{Leaf addition} ($\cG \to \cG$, $\cG' \to \cG'$): joining a new vertex to an existing vertex via an (undistinguished) edge.
\item \textbf{Edge duplication} ($\cG \to \cG$, $\cG' \to \cG'$, $\cG''\to \cG''$): duplicating an edge, that must be undistinguished in the cases of $\cG'$ and $\cG''$.
\item \textbf{Edge subdivision} ($\cG \to \cG$, $\cG' \to \cG'$, $\cG''\to \cG''$): replacing an edge, that must be undistinguished in the cases of $\cG'$ and $\cG''$, with a path of length $2$.
\item \textbf{Planar dual} ($\cG \to \cG$, $\cG' \to \cG'$, $\cG''\to \cG''$): embedding in the plane and taking the dual. As edges of the dual graph correspond to edges of the original graph, in the cases of $\cG'$ and $\cG''$ there is a canonical choice of which edge to distinguish.
\item \textbf{\boldmath $2$-sums} ($\cG' \times \cG' \to \cG'$, $\cG'' \times \cG'' \to \cG''$): taking the $2$-sum along the distinguished edges, with the final merged edge remaining distinguished.
\end{itemize}
The fact that the results of these operations lie in $\cG$, $\cG'$, or $\cG''$ is both well-known and easily checked.

We may now state a classical classification result for $\cG$. It is commonly attributed to Dirac \cite{Dirac52} and Duffin \cite{Duffin65}, even though the results of \cite{Dirac52,Duffin65} rely on a connectivity hypothesis that is absent here. The reader unsatisfied with this situation should note that the proof of \cref{lem:spb} can easily be adapted to prove \cref{lem:spa} as well.
\begin{lem}[Dirac \cite{Dirac52}, Duffin \cite{Duffin65}] \label{lem:spa}
Every connected series-parallel graph can be obtained from a single vertex vertex $K_1$ via loop addition, leaf addition, edge duplication, and edge subdivision.
\end{lem}
We will also need an analogous classification for $\cG'$. The following result is similar in spirit to \cite[Corollary 4]{Duffin65}, but as its precise statement does not have to previously appeared, we will give a proof.
\begin{lem} \label{lem:spb}
Let $(C_2, e) \in \cG'$ be the $2$-cycle with a distinguished edge, i.e.\ the multigraph with $2$ vertices and $2$ parallel edges joining them, where one is distinguished. Then every element of $\cG'$ can be obtained from $(C_2, e)$ via loop addition, leaf addition, edge duplication, and edge subdivision.
\end{lem}
\begin{proof}
Suppose not and consider a counterexample $(G, e)$ with the minimum number of edges. Since $G$ is series-parallel, it is known \cite{RS91} that it has \emph{branch-width} at most $2$, i.e.\ there exists an unrooted binary tree $T$, with its leaves corresponding to the edges of $G$, such that for every edge $e'$ of $T$ there exist at most $2$ vertices of $G$ that have incident edges corresponding to leaves on both sides of $e'$. Root $T$ at the leaf corresponding to $e$. The only element of $\cG'$ with at most $2$ edges is $C_2$, so $e(G) \geq 3$. Thus, there must exist two (non-root) leaves of $T$ that have a common parent. Suppose these leaves correspond to edges $st$ and $uv$ in $G$.

If $G$ has loops, removing them contradicts minimality, so $s \neq t$ and $u \neq v$. If $\set{s,t} = \set{u,v}$, then we have two parallel undistinguished edges, again a contradiction. Therefore $\abs{\set{s,t,u,v}} \geq 3$. By the branch-width condition, at most two elements of $\set{s,t,u,v}$ are incident to an edge in the rest of $G$, so take some $w \in \set{s,t,u,v}$ that is not. Then $w$ is only adjacent to undistinguished edges and has degree either $1$ or $2$. Either case yields a contradiction.
\end{proof}
Finally, we prove two classification results for $\cG''$.
\begin{lem} \label{lem:spc}
Every element of $\cG''$ can be obtained from $C_2$ via edge duplication and edge subdivision.
\end{lem}
\begin{proof}
This follows from \cref{lem:spb} and the observation that if a loop or leaf is ever added, a cut vertex is introduced in the graph, which can never be removed.
\end{proof}

\begin{lem} \label{lem:spd}
Every element of $\cG''$ can be obtained from copies of $C_2$ via duals and $2$-sums.
\end{lem}
\begin{proof}
Take some $(G, e)\in \cG''$; we will apply induction on the number of edges in $G$. If $e(G) \leq 2$ then we must have $G = C_2$. Otherwise, by \cref{lem:spc} we may obtain $(G, e)$ from $C_2$ via edge duplication and subdivision, which at least one operation being performed. If the first operation to be performed is edge duplication, we claim that $(G, e)$ can be written as a $2$-sum of some $(G_1,e_1), (G_2, e_2) \in \cG''$, which will finish by the inductive hypothesis as $e(G_1), e(G_2) < e(G)$. To do this, suppose the first edge duplication produces edges $e'$ and $e''$. Then, we may let $(G_1, e_1)$ be $(G, e)$ but with all edges originating from $e''$ deleted, and similarly let $(G_2, e_2)$ be $(G, e)$ but with all edges originating from $e'$ deleted.

If the first operation to be performed is edge subdivision, observe that there exists a dual graph $(G^*, e^*) \in \cG''$ that can be constructed in the same way as $(G, e)$, but with every instance of edge duplication replaced with edge subdivision, and vice versa. Since $e(G) = e(G^*)$, we conclude that $(G^*, e^*)$ falls into the case discussed above and can thus be obtained from copies of $C_2$ via duals and $2$-sums. We can then obtain $(G, e)$ as the dual of $(G^*, e^*)$.
\end{proof}

\section{Proof of Main Theorems} \label{sec:main}
To prove \cref{thm:mainv,thm:maine}, we first move to a more convenient notion of containment, which only focuses on a single layer of a hypercube. This leads to the layer-focused notion of \vocab{$\ell$-degeneracy}. After strengthening $\ell$-degeneracy to a quantitative condition that we call \vocab{$p$-degeneracy}, we define two operations, called duplication and coduplication, which are shown to turn $p$-degenerate patterns into larger $p$-degenerate patterns using a Sidorenko-like mirroring argument. The proof concludes by showing that duplication and coduplication perfectly correspond to the operations of \cref{lem:spa} and \cref{lem:spb}.

\subsection{Layers of the hypercube}
For the remainder of this paper, identify $V(Q_n)$ with $\set{0,1}^n$, the set of binary strings of length $n$. Moreover, we identify $E(Q_n)$ with the set of strings of length $n$ on the alphabet $\set{0,1,\astk}$ that contain exactly one $\astk$, corresponding to the direction of the edge.
\begin{defn}
If $a$ and $b$ are nonnegative integers, let $L_{a,b} \subseteq V(Q_{a+b})$ and $L'_{a,b} \subseteq E(Q_{a+b+1})$ be the vertices or edges, respectively, corresponding to strings with $a$ zeroes and $b$ ones.
\end{defn}
We now define a class of canonical maps between layers of the hypercube.
\begin{defn}
Suppose $a\leq a'$ and $b \leq b'$ are nonnegative integers. Let $P_{a,b,a',b'}$ be the set of strings (of length $a' + b'$) that are permutations of $s_1s_2\cdots s_{a+b} 0^{a'-a} 1^{b'-b}$, where exponents denote repetition. For $p = p_1p_2 \cdots p_{a'+b'} \in P_{a,b,a',b'}$ define $\psi_p \colon L_{a,b} \to L_{a',b'}$ to be ``the function described by $p$'', i.e.
\[\psi_p(s_1s_2 \cdots s_{a+b})_i = \begin{cases}
s_j & p_i = s_j \\
0 & p_i = 0 \\
1 & p_i = 1.
\end{cases}\]
Let $P'_{a,b,a',b'}$ be the set of strings (of length $a'+b'+1$) that are permutations of $s_1s_2 \cdots s_{a+b+1} 0^{a'-a} 1^{b'-b}$, and for $p \in P'_{a,b,a',b'}$ define $\psi_p\colon L'_{a,b} \to L'_{a',b'}$ similarly.
\end{defn}
Observe that
\[\abs{P_{a,b,a',b'}} = \frac{(a'+b')!}{(a'-a)!(b'-b)!} \quad\text{and}\quad \abs{P'_{a,b,a',b'}} = \frac{(a'+b'+1)!}{(a'-a)!(b'-b)!}.\]
\begin{defn}
Let $a \leq a'$ and $b\leq b'$ be nonnegative integers. For $X \subseteq L_{a,b}$, let $\ex(L_{a',b'}, X)$ be the size of the largest subset $S$ of $L_{a',b'}$ such that $\psi_p(X) \nsubseteq S$ for any $p \in P_{a,b,a',b'}$. Call $X$ \vocab{$\ell$-degenerate} if there is a function $f$ with $\lim_{x \to\infty} f(x) = 0$ such that $\ex(L_{a',b'}, X) \leq f(\min(a',b')) \abs{L_{a',b'}}$. For $Y \subseteq L'_{a,b}$, define $\ex(L'_{a',b'}, Y)$ and $\ell$-degeneracy similarly.
\end{defn}
While the notion of $\ell$-degeneracy will be helpful for later discussion, for the proof of \cref{thm:mainv,thm:maine} we will need a more quantitative version.
\begin{defn}
Let $a \leq a'$ and $b\leq b'$ be nonnegative integers. If $X \subseteq L_{a,b}$ and $X'\subseteq L_{a',b'}$, let
\[t(X, X') = \frac{\abs{\setmid{p \in P_{a,b,a',b'}}{\psi_p(X) \subseteq X'}}}{\abs{P_{a,b,a',b'}}}.\]
Similarly, if $Y \subseteq L'_{a,b}$ and $Y'\subseteq L'_{a',b'}$, let
\[t(Y,Y') = \frac{\abs{\setmid{p \in P'_{a,b,a',b'}}{\psi_p(Y) \subseteq Y'}}}{\abs{P'_{a,b,a',b'}}}.\]
\end{defn}

\begin{defn} \label{def:pdegen}
Call $X \subseteq L_{a,b}$ \vocab{$p$-degenerate} if there exist constants $c_1,c_2,c_3,c_4 > 0$ such that if $a' \geq a$, $b' \geq b$, and $X' \subseteq L_{a',b'}$ with $\abs{X'} = \delta \abs{L_{a',b'}}$ satisfies $\delta \geq c_1\min(a',b')^{-c_2}$, then $t(X, X') \geq c_3\delta^{c_4}$. Define $p$-degeneracy for subsets $Y \subseteq L'_{a,b}$ similarly.
\end{defn}
Note that $p$-degeneracy implies $\ell$-degeneracy with $f(x) = c_1 x^{-c_2}$.

\subsection{Duplication and coduplication}
We now describe a procedure to construct $p$-degenerate sets, using two operations $D_i$ and $D'_i$, called \vocab{duplication} and \vocab{coduplication}, respectively.
\begin{defn}
Suppose $(a,b) \neq (0,0)$ and $X\subseteq L_{a,b}$. Define
\[D_{a+b}(X) = \setmid{s00}{s0 \in X} \cup \setmid{s01}{s1 \in X} \cup \setmid{s10}{s1 \in X} \subseteq L_{a+1,b}.\]
and
\[D'_{a+b}(X) = \setmid{s01}{s0 \in X} \cup \setmid{s10}{s0 \in X} \cup \setmid{s11}{s1 \in X} \subseteq L_{a,b+1}.\]
For $i \in [a+b]$, define $D_i X$ and $D'_i X$ similarly but distinguishing the $i$th coordinate instead of the $(a+b)$-th.
\end{defn}
\begin{defn}
Suppose $Y\subseteq L'_{a,b}$. Define
\begin{multline*}
D_{a+b+1}(Y) = \setmid{s00}{s0 \in Y} \cup \setmid{s01}{s1 \in Y} \cup \setmid{s01}{s1 \in Y} \\ \cup \setmid{s0\astk}{s\astk \in Y} \cup \setmid{s\astk 0}{s\astk  \in Y} \subseteq L'_{a+1,b}.
\end{multline*}
and
\begin{multline*}
D'_{a+b+1}(Y) = \setmid{s01}{s0 \in Y} \cup \setmid{s10}{s0 \in Y} \cup \setmid{s11}{s1 \in Y} \\ \cup \setmid{s1\astk}{s\astk \in Y} \cup \setmid{s\astk s}{s\astk \in Y} \subseteq L'_{a,b+1}.
\end{multline*}
For $i \in [a+b+1]$, define $D_i Y$ and $D'_i Y$ similarly but distinguishing the $i$th coordinate instead of the $(a+b+1)$-th.
\end{defn}

\begin{lem} \label{lem:dupdegen}
Duplication and coduplication preserve $p$-degeneracy.
\end{lem}
\begin{proof}
We will only consider the vertex case, as the edge case is almost identical. Choose some $p$-degenerate $X\subseteq L_{a,b}$. Without loss of generality, it suffices to show that $D_{a+b} X$ is $p$-degenerate, since the argument for $D_i X$ is the same but with permuted coordinates, and the argument for $D'_i X$ is the same but with permuted coordinates and $0$s and $1$s swapped.

Fix some $a' \geq a+1$ and $b' \geq b$. For $p, p' \in P_{a,b,a',b'}$, write $p \sim p'$ if they agree in the locations of the $1$s and $s_1,s_2,\ldots,s_{a+b-1}$. If $p \neq p'$ and $p \sim p'$, one can construct $q \in P_{a+1,b,a',b'}$ which is the same as $p$ but with an $s_{a+b+1}$ replacing the $0$ in the location where $p'$ has an $s_{a+b}$. One can check that in this case, $\psi_q(D_{a+b} X) = \psi_p(X) \cup \psi_{p'}(X)$, so for any $X' \subseteq L_{a',b'}$, we have $\psi_q(D_{a+b} X) \subseteq X'$ if and only if $\psi_p(X) \subseteq X'$ and $\psi_{p'}(X) \subseteq X'$.

Since the map $(p, p') \mapsto q$ is reversible, we conclude that for any $X' \subseteq L_{a',b'}$, we have
\[t(D_{a+b} X, X') = \frac{1}{\abs{P_{a+1,b,a',b'}}} \sum_{T \in P_{a,b,a',b'}/\mathord{\sim}} n_T(n_T - 1),\]
where $n_T = \abs{\setmid{p \in T}{\psi_p(X) \subseteq X'}}$. By definition, we also have
\[t(X, X') = \frac{1}{\abs{P_{a,b,a',b'}}} \sum_{T \in P_{a,b,a',b'}/\mathord{\sim}} n_T.\]
Since the equivalence classes of $\sim$ have size $a'-a+1$, if we let $\beta = t(X, X')$ we have
\[\frac{1}{\abs{P_{a,b,a',b'}/\mathord{\sim}}} \sum_{T \in P_{a,b,a',b'}/\mathord{\sim}} n_T = (a'-a+1)\beta.\]
Thus, by the convexity of the function $n_T \mapsto n_T(n_T - 1)$, we find that (assuming $\beta > 0$)
\begin{align*}
t(D_{a+b} X, X') &\geq \frac{\abs{P_{a,b,a',b'}/\mathord{\sim}}}{\abs{P_{a+1,b,a',b'}}} (a'-a+1)\beta ((a'-a+1)\beta - 1) \\
&= \frac{(a'-a+1)\beta((a'-a+1)\beta - 1)}{(a'-a+1)(a'-a)} \\
&= \beta^2 \paren*{1-\frac{1/\beta-1}{a'-a}}.
\end{align*}

Now let $c_1,c_2,c_3,c_4,\delta$ be as in \cref{def:pdegen}. We know that if $\delta \geq c_1 \min(a',b')^{-c_2}$, then $\beta \geq c_3 \delta^{c_4}$. If it is further the case that 
\[a' \geq (a + 2/c_3) \delta^{-c_4} \iff \delta \geq (a+2/c_3)^{1/c_4} (a')^{-1/c_4},\] then $a' - a \geq 2/(c_3\delta^{c_4})$, so $(1/\beta - 1)/(a'-a) \leq 1/2$. This implies that $t(D_{a+b} X, X') \geq c_3^2/2 \cdot \delta^{2c_4}$. Therefore, if we let $c'_1 = \max(c_1, (a + 2/c_3)^{1/c_4})$, $c'_2 = \min(c_2, 1/c_4)$,  $c'_3 = c_3^2/2$, and $c'_4 = 2c_4$, we find that if $\delta \geq c'_1 \min(a',b')^{-c'_2}$, we have $t(D_{a+b} X, X') \geq c'_3 \delta^{c'_4}$, as desired.
\end{proof}
We now relate the graph operations of \cref{lem:spa,lem:spb} to duplication and coduplication.
\begin{thm} \label{thm:core}
If $G \in \cG$, then $X(G) \subseteq L_{e(G)-v(G)+1,v(G)-1}$ is $p$-degenerate.
If $(G, e) \in \cG'$, then $Y(G, e) \subseteq L'_{e(G)-v(G),v(G)-2}$ is $p$-degenerate.
\end{thm}
\begin{proof}
We will again only prove the vertex case. Observe that $X(K_1) = L_{0,0}$ is trivially $p$-degenerate. It suffices to show that if applying one of the four operations of \cref{lem:spa} turns $G$ into $G'$, then the $p$-degeneracy of $X(G)$ implies the $p$-degeneracy of $X(G')$.

Since a spanning tree can never contain a loop, if a loop is added then $X(G')$ is the same as $X(G)$, except with an extra $0$ in the coordinate corresponding to the added loop. It follows that for any $X'$, we have $t(X(G), X') = t(X(G'), X')$, provided both are defined. Therefore, $X(G')$ is $p$-degenerate (with the same parameters as $X(G)$). Similarly, adding a leaf adds an edge that must be present in all spanning trees, so in this case $X(G')$ is the same as $X(G)$, but with an extra $1$ in the coordinate corresponding to the added edge. Hence we are done in this case as well.

Suppose that the $i$th edge in $G$, called $e$, is duplicated to yield edges $e_1$ and $e_2$ in a graph $G'$. In this case, we claim that $X(G')$ is, under an appropriate ordering, $D_i X(G)$. This follows from the fact that the spanning trees of $G'$ are simply the spanning trees of $G$, but with any occurrence $e$ replaced with exactly one of $e_1$ and $e_2$. Therefore, the $p$-degeneracy of $X(G')$ follows from \cref{lem:dupdegen}.

Similarly, if the $i$th edge in $G$, called $e$, is subdivided to yield edges $e_1$ and $e_2$ in a graph $G'$, then the spanning trees of $G'$ are the spanning trees of $G$, except with exactly one of $e_1$ and $e_2$ added if the original tree doesn't contain $e$ and both $e_1$ and $e_2$ added if the original tree contains $e$. Therefore $X(G')$ is, under an appropriate ordering, the same as $D'_i X(G)$, so we are again done by \cref{lem:dupdegen}.
\end{proof}

\subsection{From layers to the hypercube}\label{sec:mainproof}
Finally, we move from layers to the full hypercube. By summing over all layers, it is easy to see that $\ell$-degeneracy implies $h$-degeneracy. In fact, this proves something even stronger: that the pattern is degenerate even if we restrict the notion of embedding. Indeed, call a embedding $\psi\colon Q_d \to Q_n$ \vocab{oriented} if it preserves the natural partial orders on $Q_d$ and $Q_n$, or equivalently, if there are constants $a'$ and $b'$ such that $\psi(L_{a,b}) \subseteq L_{a+a',b+b'}$ for all $a,b \geq 0$ with $a+b=d$. Then every $\ell$-degenerate pattern is $h$-degenerate even when only oriented embeddings are considered.

Applying a similar argument (which becomes more complicated due to the quantitative bounds) to the $p$-degenerate $X(G)$ and $Y(G, e)$ yields the strongest versions of the main results of this paper:
\begin{thm} \label{thm:fullv}
For every $G \in \cG$ there exist constants $c_1,c_2,c_3,c_4 > 0$ such that if $n \geq e(G)$, $S\subseteq V(Q_n)$ has size $\delta 2^n$, and $\delta \geq c_1 n^{-c_2}$, then at least a $c_3 \delta^{c_4}$-proportion of oriented embeddings $\psi\colon Q_{e(G)} \to Q_n$ satisfy $\psi(X(G)) \subseteq S$.
\end{thm}
\begin{thm} \label{thm:fulle}
For every $(G, e) \in \cG'$ there exist constants $c_1,c_2,c_3,c_4 > 0$ such that if $n \geq e(G)-1$, $S\subseteq E(Q_n)$ has size $\delta e(Q_n)$, and $\delta \geq c_1 n^{-c_2}$, then at least a $c_3 \delta^{c_4}$-proportion of oriented embeddings $\psi\colon Q_{e(G)} \to Q_n$ satisfy $\psi(Y(G, e)) \subseteq S$.
\end{thm}
These evidently imply \cref{thm:mainv,thm:maine} as they imply that $\ex(Q_n, X(G)) \leq c_1 n^{-c_2} 2^n$ and $\ex(Q_n, Y(G, e)) \leq c_1 n^{-c_2} e(Q_n)$, respectively.

As the proofs of \cref{thm:fullv,thm:fulle} are almost identical, we will only prove \cref{thm:fullv}.
\begin{proof}[Proof of \cref{thm:fullv}]
By \cref{thm:core}, $X(G)$ is $p$-degenerate, and let $c'_1,c'_2,c'_3,c'_4$ be the constants in \cref{def:pdegen}. Without loss of generality assume $c'_4 \geq 1$. We commit now to choosing $c_2 = c'_2$ and $c_4 = c'_4$, while $c_1$ and $c_3$ are to be determined later.

Suppose $X(G) \subseteq L_{a,b} \subseteq V(Q_{a+b})$ and take some $S \subseteq V(Q_n)$ with $\abs{S} = \delta 2^n$. It suffices to show that either $\delta = O(n^{-c_2})$ or an $\Omega(\delta^{c_4})$-proportion of oriented embeddings $\psi$ satisfy $\psi(X(G)) \subseteq S$.

By adjusting $c_1$, we may take $n$ to be arbitrarily large. Assume $n \geq 3\max(a,b)$ and for $0 \leq k \leq n$ define
\begin{align*}
\alpha_k &= \begin{cases} \max(\abs{S \cap L_{k, n-k}}/\binom{n}{k} - c'_1 \min(k, n-k)^{-c_2}, 0) & n/3 \leq k \leq 2n/3 \\
0 & \text{else},
\end{cases} \\
\beta_k &= \begin{cases} t(X(G), S \cap L_{k,n-k}) & n/3 \leq k \leq 2n/3 \\
0 & \text{else}.
\end{cases}
\end{align*}
and note that $p$-degeneracy tells us that $\beta_k \geq c'_3 \alpha_k^{c_4}$. Thus, since $c_4 \geq 1$, by Jensen's inequality we conclude that
\[\sum_{0 \leq k \leq n} \frac{\binom{n}{k}}{2^n} \beta_k \geq c'_3 \paren*{\sum_{0 \leq k \leq n} \frac{\binom{n}{k}}{2^n} \alpha_k}^{c_4}.\]
As such, we either have $\sum_{0 \leq k \leq n} \binom{n}{k}/2^n \cdot \alpha_k \leq \delta/2$ or $\sum_{0 \leq k \leq n} \binom{n}{k}/2^n \cdot \beta_k \geq c'_3(\delta/2)^{c_4}$.

In the former case, we observe that
\[\delta/2 \geq \sum_{0 \leq k \leq n} \frac{\binom{n}{k}}{2^n} \alpha_k \geq \frac{1}{2^n}\sum_{n/3 \leq k \leq 2n/3} \paren*{\abs{S \cap L_{k,n-k}} - c'_1 (n/3)^{-c_2} \binom{n}{k}},\]
so
\begin{align*}
\delta &\leq \frac{1}{2^n}\sum_{k < n/3 \text{ or } k > 2n/3} \binom{n}{k} + \frac{1}{2^n} \sum_{n/3 \leq k \leq 2n/3} \abs{S \cap L_{k,n-k}} \\
&\leq \frac{1}{2^n} \sum_{k < n/3 \text{ or } k > 2n/3} \binom{n}{k} + \delta/2 + \frac{c'_1 (n/3)^{-c_2}}{2^n} \sum_{n/3 \leq k \leq 2n/3}  \binom{n}{k} \\
&\leq \delta/2 + \frac{1}{2^n} \sum_{k < n/3 \text{ or } k > 2n/3} \binom{n}{k} + c'_1 (n/3)^{-c_2}.
\end{align*}
The second term is well-known to be $O(0.95^n)$, so we have $\delta \leq \delta/2 + O(n^{-c_2})$, as desired.

In the latter case, we observe that every oriented embedding $Q_{a+b} \to Q_n$ must send $L_{a,b}$ to some $L_{k, n-k}$ with $a \leq k \leq n-b$; for a given $k$, such oriented embeddings are described perfectly by the elements of $P_{a,b,k,n-k}$, such that the restriction of the oriented embedding corresponding to some $p \in P_{a,b,k,n-k}$ to $L_{a,b}$ is precisely $\psi_p$. Hence, the proportion of oriented embeddings $\psi$ with $\psi(X(G))\subseteq S$ is
\[\frac{\sum_{a \leq k \leq n-b} \abs{P_{a,b,k,n-k}} \cdot  t(X(G), S\cap L_{k, n-k})}{\sum_{a \leq k \leq n-b} \abs{P_{a,b,k,n-k}}}.\]
We may now compute
\[\sum_{a \leq k \leq n-b} \abs{P_{a,b,k,n-k}} = \sum_{a \leq k \leq n-b} \frac{n!}{(k-a)!(n-k-b)!} = 2^{n-a-b} \frac{n!}{(n-a-b)!} = O(n^{a+b}2^n),\]
while for $n/3 \leq k \leq 2n/3$ we have
\[\frac{\abs{P_{a,b,k,n-k}}}{\binom{n}{k}} = \frac{k! (n-k)!}{(k-a)! (n-k-b)!} = \Omega(n^{a+b}).\]
Therefore, there is some constant $C$ such that
\[\frac{\sum_{a \leq k \leq n-b} \abs{P_{a,b,k,n-k}} \cdot  t(X(G), S\cap L_{k, n-k})}{\sum_{a \leq k \leq n-b} \abs{P_{a,b,k,n-k}}} \geq \frac{Cn^{a+b} \sum_{n/3\leq k \leq 2n/3} \binom{n}{k} \beta_k}{n^{a+b} 2^n} \geq Cc'_3(\delta/2)^{c_4},\]
as desired.
\end{proof}

\section{Additional Results} \label{sec:add}
\subsection{Duality}
If $a$ and $b$ are nonnegative integers and $X \subseteq L_{a,b}$, let $X^* \subseteq L_{b,a}$ be the set obtained from $X$ by swapping $0$s and $1$s. By reversing the cube, it is easy to see that $X$ is $h$-degenerate if and only if $X^*$ is, even under the stricter notion of oriented embedding used in \cref{sec:mainproof}. Moreover, the same is true for $Y \subseteq L'_{a,b}$ and $Y^* \subseteq L'_{b,a}$, defined similarly.

This symmetry is also evident in \cref{thm:mainv,thm:maine}, as shown by the following observation.
\begin{prop} \label{prop:duality}
Let $G\in \cG$ and let $G^*$ be a dual of $G$. Then $X(G)^* = X(G^*)$. Moreover, if $(G, e) \in \cG'$ and $(G^*, e^*)$ is a dual, then $Y(G, e)^* = Y(G^*, e^*)$.
\end{prop}
\begin{proof}
This result follows from the classical fact that the complement of a spanning tree in a planar graph determines a spanning tree of the dual graph.
\end{proof}

\subsection{An alternate graph-theoretic perspective} \label{sec:pj}
If we only care about the graph structure of $h$-degenerate edge patterns, which must be $g$-degenerate, we can find an alternate description that only uses basic graph operations. To better discuss this notion, for connected series-parallel $G$ and $e \in E'(G)$ let $H(G, e)$ be the induced (bipartite) subgraph on $Q_{e(G)-1}$ on vertex set $X(G\bslash e) \cup X(G\setminus e)$. By construction, $Y(G, e)$ gives an embedding of $H(G, e)$ in $Q_{e(G)-1}$, so $H(G, e)$ is $g$-degenerate.
\begin{defn}
If $H_1$ and $H_2$ be connected bipartite graphs with bipartitions $(A_1, B_1)$ and $(A_2, B_2)$, let the \vocab{product-join} of $H_1$ and $H_2$ along $A_1$ and $A_2$ be the (connected) bipartite graph with bipartition $(A_1 \times A_2, (A_1 \times B_2) \sqcup (B_1 \times A_2))$ where $(a_1,a_2) \in A_1 \times A_2$ is connected to $(a_1,b_2)$ for all $b_2$ adjacent to $a_2$ in $H_2$ and $(b_1, a_2)$ for all $b_1$ adjacent to $a_2$ in $H$. Call a graph \vocab{product-join-reducible} if it can be obtained from copies of $K_2$ via taking connected subgraphs and product-joins.
\end{defn}
\begin{thm} \label{thm:pjr}
The following are equivalent for a connected graph $H$:
\begin{parts}
\item $H$ is product-join-reducible.
\item $H$ is a subgraph of $H(G, e)$ for some $(G, e) \in \cG'$.
\item $H$ is a subgraph of $H(G, e)$ for some $(G, e) \in \cG''$.
\end{parts}
\end{thm}
\begin{rmk}
One way to interpret this result is through recognizing the product-join of $H_1$ and $H_2$ as the subgraph of the \vocab{box product} $H_1 \mathbin{\square} H_2$ obtained by deleting $B_1 \times B_2$. The box product of $n$ copies of $K_2$ is $Q_n$, so the product-join-reducible graphs are simply those obtained through the same box products, except that at each intermediate stage the graph is pruned to lie in a single layer of the hypercube.
\end{rmk}
To prove \cref{thm:pjr}, we first need a few preliminary results about $H(G, e)$.
\begin{prop} \label{prop:g0}
Let $(G, e) \in \cG'$ and let $G_0$ be the block of $G$ containing $e$. Then $(G_0, e) \in \cG'$ and $H(G, e)$ consists of some number of disjoint copies of $H(G_0, e)$.
\end{prop}
\begin{proof}
Since $e$ is not a bridge or a loop, $G_0$ is $2$-connected and thus has no bridges or loops, implying that $(G_0, e) \in \cG'$.

Let $G'$ be $G$ but with $G_0$ contracted to a point. Spanning trees of $G$ correspond uniquely to pairs of spanning trees of $G_0$ and $G'$, respectively, so $X(G\bslash e) = X(G_0\bslash e) \times X(G')$ and $X(G\setminus e) = X(G_0\setminus e) \times X(G')$. It follows that $H(G, e)$ consists of $\abs{X(G')}$-many disjoint copies of $H(G_0, e)$.
\end{proof}
\begin{prop} \label{lem:gluing}
Let $(G_1,e_1),(G_2,e_2)\in \cG'$ be such that $H(G_1, e_1)$ and $H(G_2,e_2)$ are connected. Let $(G, e)$ is a $2$-sum of $(G_1,e_1)$ and $(G_2,e_2)$, then $H(G, e)$ is the product-join of $H(G_0,e_1)$ and $H(G,e_2)$ along $X(G_1 \bslash e_1)$ and $X(G_2 \bslash e_2)$.
\end{prop}
\begin{proof}
Observe that $G \bslash e$ is a $1$-sum of $G_1 \bslash e_1$ and $G_2\bslash e_2$, so $X(G\bslash e) = X(G_1 \bslash e_1) \times X(G_2 \bslash e_2)$. Moreover, we have
\[X(G \setminus e) = (X(G_1 \bslash e_1) \times X(G_2 \setminus e_2)) \sqcup (X(G_1 \setminus e_1) \times X(G_2 \bslash e_2)),\]
where the two products represent spanning trees where the path connecting the endpoints of $e$ passes through $G_2$ and $G_1$, respectively. It suffices to verify that the adjacency relation on $H(G, e)$ is exactly that given by the product-join, which is straightforward.
\end{proof}

\begin{prop} \label{prop:fullpj}
A graph can be obtained from copies of $K_2$ through product-joins if and only if it is isomorphic to $H(G, e)$ for some $(G, e) \in \cG''$.
\end{prop}
\begin{proof}
We make two observations:
\begin{itemize}
\item If $G = C_2$ and $e\in E(G)$ is arbitrary, then $H(G, e) = K_2$.
\item If $(G^*, e^*)$ is a dual of some $(G, e) \in \cG''$, the canonical isomorphism between $H(G, e)$ and $H(G^*, e^*)$ sends $X(G\bslash e)$ to $X(G^*\setminus e^*)$ and $X(G\setminus e)$ to $X(G^*\bslash e^*)$.
\end{itemize}
With these two observations, the result follows from combining \cref{lem:spd,lem:gluing}, as any sequence of product-joins can be converted to a sequence of duals and $2$-sums, and vice versa.
\end{proof}

\begin{proof}[Proof of \cref{thm:pjr}]
The equivalence of (b) and (c) follows from \cref{prop:g0}. The implication $\text{(c)} \implies \text{(a)}$ follows from \cref{prop:fullpj}. To show $\text{(a)} \implies \text{(c)}$, observe that since a product-join of connected subgraphs of $H_1$ and $H_2$ is a subgraph of a product-join of $H_1$ and $H_2$, any product-join-reducible $H$ must be the connected subgraph of a graph obtained by taking product joins of copies of $K_2$. Applying \cref{prop:fullpj} finishes.
\end{proof}

\subsection{Relation to earlier results} \label{sec:cor}
\begin{figure}
\subcaptionbox{\label{fig:v}}[0.25\linewidth]{%
\begin{tikzpicture}[scale=1.25]
\draw (144:1) -- (72:1) to[bend left] (144:1) to[bend left] node[shift={(108:8pt)}]{$a_1$} (72:1);
\draw (72:1)to[bend left] node[shift={(36:8pt)}]{$a_2$} (0:1) to[bend left](72:1);
\draw (0:1)to[bend left=10](-72:1)to[bend left=10](0:1)to[bend left=30] node[shift={(-36:8pt)}]{$a_3$} (-72:1)to[bend left=30](0:1);
\draw (-72:1)-- node[shift={(-108:8pt)}]{$a_4$} (-144:1);
\path[every node/.style=vtx] (144:1)node{} (72:1)node{} (0:1)node{} (-72:1)node{} (-144:1)node{};
\path (144:1 |- 0,0) node{\rotatebox{90}{$\cdots$}};
\path (-1.5,0) (1.5,0);
\end{tikzpicture}}%
\subcaptionbox{\label{fig:e}}[0.25\linewidth]{%
\begin{tikzpicture}[scale=1.25]
\draw[ultra thick] (-150:1)--(150:1);
\draw (150:1) -- (90:1) to[bend left] (150:1) to[bend left] node[shift={(120:8pt)}]{$a_1$} (90:1);
\draw (90:1)to[bend left] node[shift={(60:8pt)}]{$a_2$} (30:1) to[bend left](90:1);
\draw (30:1)to[bend left=10](-30:1)to[bend left=10](30:1)to[bend left=30] node[shift={(0:8pt)}]{$a_3$} (-30:1)to[bend left=30](30:1);
\draw (-30:1)-- node[shift={(-60:8pt)}]{$a_4$} (-90:1);
\path (150:1)node[vtx]{} (90:1)node[vtx]{} (30:1)node{} (-30:1)node[vtx]{} (-90:1)node[vtx]{} -- node{\rotatebox{-30}{$\cdots$}} (-150:1)node[vtx]{};
\path (-1.5,0) (1.5,0);
\end{tikzpicture}}%
\subcaptionbox{\label{fig:theta}}[0.4\linewidth]{%
\begin{tikzpicture}[every node/.style=vtx]
\begin{scope}[xscale=0.8,yscale=0.5]
\draw (0,0)node{} -- (1,0)node{} -- (2,0)node{} -- (3,0)node{} -- (4,0)node{} -- (3,1)node{} -- (2,1)node{} -- (1,1)node{} -- (0,0) -- (1,-1)node{} -- (2,-1)node{} -- (3,-1)node{} -- (4,0);
\end{scope}
\begin{scope}[shift={(2.4,0.5)},rotate=40,xscale=0.8,yscale=0.5]
\draw (0,0)node{} -- (1,0)node{} -- (2,0)node{} -- (3,0)node{} -- (4,0)node{} -- (3,1)node{} -- (2,1)node{} -- (1,1)node{} -- (0,0) -- (1,-1)node{} -- (2,-1)node{} -- (3,-1)node{} -- (4,0);
\end{scope}
\end{tikzpicture}}%
\caption{\subref{fig:v} Multigraph used in the proof of \cref{cor:alon}. Labels represent edge multiplicities.\quad
\subref{fig:e} Multigraph used in the proof of \cref{cor:partite}. Labels represent edge multiplicities and the distinguished edge highlighted.\quad
\subref{fig:theta} A $1$-sum of two copies of $\Theta_4$ that does not have a partite representation.}
\end{figure}
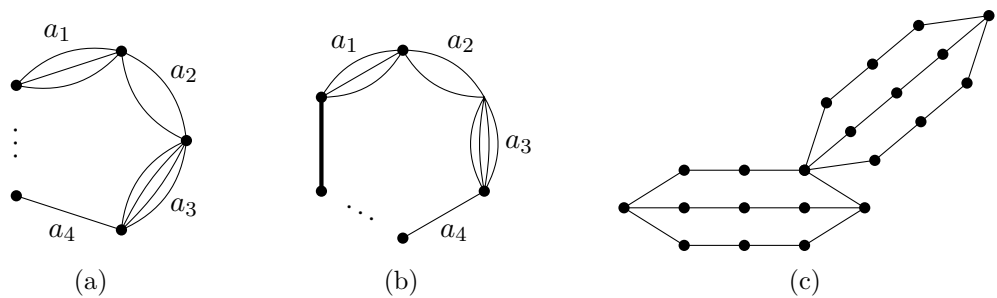

At this point, we now have the tools to demonstrate how a number of earlier results are special cases of \cref{thm:mainv,thm:maine}. First, we recover a result of Alon on $h$-degenerate vertex patterns.
\begin{cor}[{\cite[Lem.~2.2]{Alon24}}] \label{cor:alon}
Let $k$ and $a_1,\ldots,a_k$ be positive integers and let $d = a_1 + \cdots + a_k$. Let $X \subseteq L_{d-k+1,k-1}$ be the set of $\sum_{i \in [k]} \prod_{j \neq i} a_j$ length-$n$ binary strings such that, if divided into blocks of length $a_1,\ldots,a_k$, all but one of the blocks contains exactly one $1$ and the remaining block is all $0$s. Then $X$ is $h$-degenerate.
\end{cor}
\begin{proof}
This $X$ is $X(G)$, where $G$ consists of $a_1$ parallel edges, $a_2$ parallel edges, etc., all arranged in a cycle. This graph $G$ is depicted in \cref{fig:v}.
\end{proof}

We now turn to edge patterns. Conlon \cite{Conlon10} proved that every graph with a partite representation is $g$-degenerate, which we will now define.
\begin{defn} \label{def:partite}
If $k$ and $a_1,\ldots,a_k$ be positive integers and $n = a_1 + \cdots + a_k$, then let $Y_{a_1,\ldots,a_k} \subseteq L'_{n-k,k-1}$ be the set of $k\prod_{i \in [k]} a_i$ length-$n$ strings such that, if divided into blocks of length $a_1,\ldots,a_k$, all but one of the blocks is all $0$s except for a single $1$, and the remaining block is all $0$s except for a single $\astk$. A graph has a \vocab{partite representation} if it can be expressed as a subgraph of $Y_{a_1,\ldots,a_k}$ for some choice of $a_1,\ldots,a_k$.
\end{defn}
\begin{cor}[{\cite[Thm.~2.1]{Conlon10}}] \label{cor:partite}
$Y_{a_1,\ldots,a_k}$ is $h$-degenerate. In particular, any graph with a partite representation is $g$-degenerate.
\end{cor}
\begin{proof}
Here, $Y_{a_1,\ldots,a_k} = Y(G, e)$, where $(G, e)$ is the multigraph consisting of the distinguished edge $e$, $a_1$ parallel edges, $a_2$ parallel edges, etc., all arranged in a cycle. This multigraph is depicted in \cref{fig:e}.
\end{proof}

\begin{cor}[{\cite[Thm.~1]{Axen23}}] \label{cor:ax}
Any graph whose blocks all have partite representations is $g$-degenerate.
\end{cor}
\begin{proof}
As it is easy to see that the disjoint union of two $g$-degenerate graphs is $g$-degenerate, it suffices to consider $1$-sums of connected graphs with partite representations. By \cref{thm:pjr,cor:partite} all connected graphs with partite representations are product-join-reducible, so it suffices to show that a $1$-sum of product-join-reducible graphs is itself product-join-reducible.

To see this, suppose that $H_1$ and $H_2$ are connected bipartite graphs with respective bipartitions $(A_1, B_1)$ and $(A_2, B_2)$. Then, for any vertices $v_1 \in A_1$ and $v_2 \in A_2$, the $1$-sum of $H_1$ and $H_2$ identifying $v_1$ and $v_2$ is precisely the induced subgraph of the product-join of $H_1$ and $H_2$ along $A_1$ and $A_2$, on the vertex set
\[(\set{v_1} \times (A_2 \sqcup B_2)) \cup ((A_1 \sqcup B_1) \times \set{v_2}).\]
Since connected subgraphs of product-join-reducible graphs are (by definition) product-join-reducible, we are done.
\end{proof}

We conclude this section with an example of a $g$-degenerate graph that is not covered by \cref{cor:ax}, answering a question of Axenovich \cite{Axen23} in the negative.
\begin{examp} \label{ex:counter}
Let $\Theta_4$ be the theta graph consisting of $3$ internally disjoint paths of length $4$ with the same endpoints. It is a result of Marquardt \cite{Mar22} (see also \cite{Axen23}) that $\Theta_4$ has a partite representation, but there exists a $1$-sum of two copies of $\Theta_4$ (shown in \cref{fig:theta}) that does not have a partite representation. By the argument in \cref{cor:ax}, we conclude that there is a product-join of two copies of $\Theta_4$ that does not have a partite representation. However, $\Theta_4$ is $2$-connected, and it is straightforward to show that the product-join of two $2$-connected graphs is $2$-connected. Hence this product-join is $2$-connected and cannot be built from $1$-sums of smaller graphs.
\end{examp}

\subsection{Quantitative matters} \label{sec:quant}
In \cite{JT10}, Johnson and Talbot introduced the quantity $L(d)$, defined to be the size of the largest $h$-degenerate subset of $V(Q_d)$ (see \cref{sec:containment}). In \cite{Alon24}, Alon obtained improved lower and upper bounds on $L(d)$, which match for $d \leq 5$. Here, we further improve the lower bound.
\begin{prop} \label{prop:fib}
For $d \geq 0$, the maximum number of spanning trees of any connected series-parallel graph with $d$ edges is the Fibonacci number $F_{d+1}$ (indexed so that $F_1=F_2=1, F_3=2,\ldots$). In particular, $L(d) \geq F_{d+1}$.
\end{prop}
\cref{prop:fib} matches the known values of $L(d)$ for $d \leq 5$ and gives improved lower bounds for $d \geq 6$. However, there is still an exponential gap between $F_{d+1}$, which is $\Theta(\varphi^d)$ for $\varphi = \frac{1}{2}(\sqrt{5}+1) \approx 1.61803$ the golden ratio, and the best upper bound, which is $O(2^d/\sqrt{d})$.
\begin{ques}
Is there some $\eps > 0$ such that $L(d) = O((2-\eps)^d)$?
\end{ques}
\begin{figure}
\begin{tikzpicture}[scale=0.5,yscale=.866]
\begin{scope}
\draw (0,0)node[vtx]{};
\node at (0,-1) {$G_0$};
\end{scope}
\begin{scope}[shift={(3,0)}]
\draw (0,0)node[vtx]{} -- (1,2)node[vtx]{};
\node at (0.5,-1) {$G_1$};
\end{scope}
\begin{scope}[shift={(7,0)}]
\draw (0,0)node[vtx]{} to[bend right] (1,2)node[vtx]{} to[bend right] (0,0);
\node at (0.5,-1) {$G_2$};
\end{scope}
\begin{scope}[shift={(11,0)}]
\draw (0,0)node[vtx]{} -- (1,2)node[vtx]{} -- (2,0)node[vtx]{} -- (0,0);
\node at (1,-1) {$G_3$};
\end{scope}
\begin{scope}[shift={(16,0)}]
\draw (1,2)node[vtx]{} to[bend right] (2,0)node[vtx]{} to[bend right] (1,2) -- (0,0)node[vtx]{} -- (2,0);
\node at (1,-1) {$G_4$};
\end{scope}
\begin{scope}[shift={(21,0)}]
\draw (1,2)node[vtx]{} -- (2,0)node[vtx]{} -- (0,0)node[vtx]{} -- (1,2) -- (3,2)node[vtx]{} -- (2,0);
\node at (1.5,-1) {$G_5$};
\end{scope}
\begin{scope}[shift={(27,0)}]
\draw (1,2)node[vtx]{} -- (2,0)node[vtx]{} -- (0,0)node[vtx]{} -- (1,2) -- (3,2)node[vtx]{} to[bend right] (2,0) to[bend right] (3,2);
\node at (1.5,-1) {$G_6$};
\end{scope}
\end{tikzpicture}
\caption{Graphs used in the proof of \cref{prop:fib}.} \label{fig:fib}
\end{figure}
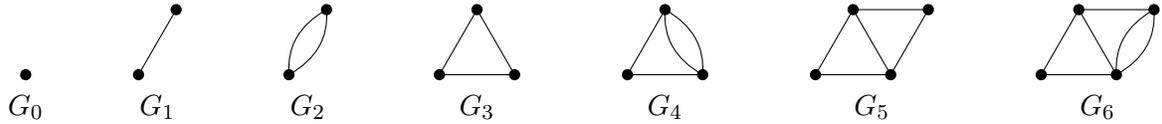
\begin{proof}[Proof of \cref{prop:fib}]
For the bound, we argue by induction on $d$, with the $d = 0$ case being trivial. For any connected series-parallel graph $G$ with $d \geq 1$ edges, by \cref{lem:spa} there must exist a series-parallel graph $G'$ with $d-1$ edges such that $G$ is obtained from $G'$ by either
\begin{itemize}
\item adding a loop;
\item adding a leaf;
\item duplicating an edge $e$ into edges $e_1$ and $e_2$;
\item subdividing an edge $e$ into edges $e_1$ and $e_2$.
\end{itemize}
In the first two cases it is easy to see that $\abs{X(G)} = \abs{X(G')}$, so we are done by the inductive hypothesis. In the third case, note that the spanning trees of $G$ without $e_2$ are in bijection with the spanning trees of $G'$, while the spanning trees of $G$ with $e_2$ are in bijection with the spanning trees of $G'$ containing $e$, which are in turn in bijection with the spanning trees of $G'\bslash e$. Hence
\[\abs{X(G)} = \abs{X(G')} + \abs{X(G'\bslash e)} \leq F_d + F_{d-1} = F_{d+1},\]
where we have used the inductive hypothesis. By similar logic, in the fourth case, we have
\[\abs{X(G)} = \abs{X(G')} + \abs{X(G'\setminus e)} \leq F_d + F_{d-1} = F_{d+1},\]
completing the induction.

To show that $F_{d+1}$ is achievable, consider the series of graphs $G_0, G_1,\ldots$ depicted in \cref{fig:fib}. It is true that $\abs{X(G_0)} = \abs{X(G_1)} = 1 = F_1 = F_2$. Moreover, observe that for even $d \geq 2$, the graph $G_d$ is obtained by duplicating an edge $e$ in $G_{d-1}$ such that $G_{d-1} \bslash e \cong G_{d-2}$. Moreover, for odd $d \geq 3$, the graph $G_d$ is obtained by subdividing an edge $e$ in $G_{d-1}$ such that $G_{d-1} \setminus e \cong G_{d-2}$. Therefore for $n \geq 2$ we have $\abs{X(G_d)} = \abs{X(G_{d-1})} + \abs{X(G_{d-2})}$, which shows that $\abs{X(G_d)} = F_{d+1}$ by induction.
\end{proof}
\begin{table} {\small
\begin{tabular}{CC@{\qquad}CC@{\qquad}CC@{\qquad}CC@{\qquad}CC@{\qquad}CC}  \toprule
d & m(d) & d & m(d) & d & m(d) & d & m(d) & d & m(d) & d & m(d)\\ \midrule
1 & 1 & 6 & 24 & 11 & 343 & 16 & 4480 & 21 & 55296 & 26 & 665496 \\
2 & 2 & 7 & 42 & 12 & 576 & 17 & 7410 & 22 & 90792 & 27 & 1088640 \\
3 & 4 & 8 & 72 & 13 & 960 & 18 & 12240 & 23 & 150336 & 28 & 1783296 \\
4 & 8 & 9 & 122 & 14 & 1608 & 19 & 20202 & 24 & 248832 & 29 & 2915244 \\
5 & 14 & 10 & 204 & 15 & 2680 & 20 & 33552 & 25 & 406944 & 30 & 4763880 \\
\bottomrule
\end{tabular}}
\caption{Values for $m(d)$ for $1 \leq d \leq 30$.} \label{tab:md}
\end{table}
Faced with this surprisingly elegant answer, one might wonder about the maximum of $\abs{Y(G, e)}$ over $(G, e) \in \cG'$ where $e(G)=d+1$; call this quantity $m(d)$. However, $m(d)$ is not so simply described; values for small $d$ are given in \cref{tab:md}, and the graphs that achieve them do not satisfy any nice pattern. However, we can still pin down the exponential growth rate of $m(d)$.
\begin{prop}
For $d \geq 1$, we have $F_{d+2} - 1 \leq m(d) \leq dF_{d+2}/2$.
\end{prop}
\begin{proof}
For the lower bound, by the construction in the proof of \cref{prop:fib} there exists a $2$-connected series-parallel graph $G_{d+1}$ with $d+1$ edges and $F_{d+2}$ spanning trees. Choose $e \in E(G_{d+1})$ arbitrarily. Then $(G,e)\in \cG''$, so by \cref{prop:fullpj} the graph $H(G, e)$ is connected. The number of vertices of $H(G, e)$ is $\abs{X(G_{d+1} \bslash e)} + \abs{X(G_{d+1} \setminus e)} = \abs{X(G_{d+1})} = F_{d+2}$, so we conclude that $\abs{Y(G_{d+1}, e)} \geq F_{d+2} - 1$.

For the upper bound, note that for $(G, e) \in \cG'$, the graph $H(G, e)$ has $\abs{X(G)} \leq F_{d+2}$ vertices. However, it is also a subgraph of $Q_d$, so every vertex has degree at most $d$. Thus $H(G, e)$ has at most $dF_{d+2}/2$ edges, as desired.
\end{proof}

\subsection{Completeness in small cases} \label{sec:small}
\begin{figure}
\subcaptionbox{\label{fig:smallv}}[0.3\linewidth]{%
\begin{tikzpicture}[scale=1.5]
\draw (120:1) -- (-120:1) to[bend left=20] node[left]{$I_0 \sqcup I_1$} (120:1) to[bend left=20] (-120:1);
\draw (120:1)to[bend left=10] node[shift={(60:8pt)}]{$I_2$} (0:1) to[bend left=10](120:1);
\draw (0:1)to[bend left=10](-120:1)to[bend left=10](0:1)to[bend left=30] node[shift={(-60:8pt)}]{$I_3$} (-120:1)to[bend left=30](0:1);
\path[every node/.style=vtx] (120:1)node{} (0:1)node{} (-120:1)node{};
\path (1.5,0);
\end{tikzpicture}}%
\subcaptionbox{\label{fig:smalle}}[0.3\linewidth]{%
\begin{tikzpicture}[scale=1.5]
\draw (-120:1) to[bend left=10] (120:1) to[bend left=10] (-120:1) to[bend left=30] node[left]{$I_0 \sqcup I_1$} (120:1);
\draw[ultra thick] (120:1) to[bend left=30] (-120:1);
\draw (120:1)to[bend left=10] node[shift={(60:8pt)}]{$I_2$} (0:1) to[bend left=10](120:1);
\draw (0:1)to[bend left=10](-120:1)to[bend left=10](0:1)to[bend left=30] node[shift={(-60:8pt)}]{$I_3$} (-120:1)to[bend left=30](0:1);
\path[every node/.style=vtx] (120:1)node{} (0:1)node{} (-120:1)node{};
\path (1.5,0);
\end{tikzpicture}}%
\caption{Multigraphs used in the proofs of \subref{fig:smallv} \cref{prop:smallv}, and
\subref{fig:smalle} \cref{prop:smalle}. The distinguished edge in \subref{fig:smalle} is highlighted, and labels give the corresponding indices of a set of parallel edges. The location of the $I_0$ edges is irrelevant to the argument.}
\end{figure}
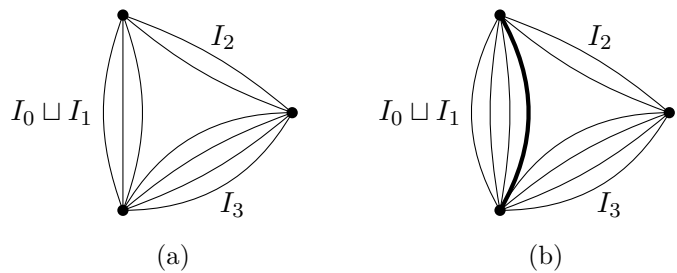
In this section, we show that in certain small cases, \cref{thm:core} give a complete description of $\ell$-degenerate vertex and edge patterns. The constructions in this section are the same as those studied in \cite{EIL24,GM24+}.

This result for vertex patterns was first stated, albeit with significantly different language, as Proposition~1.4 in \cite{Alon24}. We give a proof for completeness.
\begin{prop}[{\cite[Prop.\ 1.4]{Alon24}}] \label{prop:smallv}
If $b \leq 2$ and $X \subseteq L_{a,b}$ is $\ell$-degenerate, then $X \subseteq X(G)$ for some $G \in \cG$ (with $a+b$ edges and $b+1$ vertices).
\end{prop}
\begin{proof}
The cases $b = 0,1$ are trivial, as $L_{a,0} = X(G)$ where $G$ is the multigraph with one vertex and $a$ loops, and $L_{a,1} = X(G)$, where $G$ is the graph with series-parallel graph with two vertices and $a+1$ parallel edges connecting them.

Consider $X \subseteq L_{a,2}$ for some $a \geq 0$. For some $a' \geq a$ and $b' \geq 2$, define a random subset $S \subseteq L_{a',b'}$ as follows: randomly choose vectors $v_1,\ldots,v_{a'+b'} \in \setf_2^{b'}$, and include a string $s\in L_{a',b'}$ in $S$ if and only if the vectors corresponding the $1$s in $s$ form a basis of $\setf_2^{b'}$. Since the probability that a uniformly random $b' \times b'$ matrix over $\setf_2$ is invertible is $\prod_{i=1}^{b'} (1-2^{-i}) > \prod_{i=1}^\infty (1-2^{-i}) > 0.2887$, the expected value of $\abs{S}$ is greater than $0.2887 \abs{L_{a',b'}}$. Hence by $\ell$-degeneracy there exists some $a'\geq a$, $b' \geq 2$, choice of $v_1,\ldots,v_{a'+b'}$, and choice of $p \in P_{a,2,a',b'}$ such that $\psi_p(X) \subseteq S$.

Suppose the $1$s in $p$ occur in positions $i_1,\ldots,i_{b'-2}$ and $s_k$ occurs at position $j_k$. If $v_{i_1},\ldots,v_{i_{b'-2}}$ are linearly dependent, then the image of $\psi_p$ is disjoint from $S$, a contradiction. Thus they span a subspace $W \subseteq \setf_2^{b'}$ of codimension $2$. For every $k \in [a+2]$ let $w_k$ be the image of $v_{j_k}$ in the $2$-dimensional space $\setf_2^{b'}/W$. This induces a partition $[a+2] = I_0 \sqcup I_1 \sqcup I_2 \sqcup I_3$, where $I_0$ consists of the $k$ for which $w_k = 0$ and $I_1,I_2,I_3$ consist of the $k$ for which $w_k$ is equal to each of the three nonzero vectors in $\setf_2^{b'}/W$. If $s \in X$ has $1$s in locations $k_1$ and $k_2$, then $w_{k_1}$ and $w_{k_2}$ must be nonzero and distinct. It follows that $s \in X(G)$, where $G$ is the graph shown in \cref{fig:smallv}.
\end{proof}
A similar argument works in the case of edges. This generalizes an argument in \cite{GM24+} showing that $L'_{1,1}$ minus an edge is not $\ell$-degenerate, a key lemma in showing that $C_{10}$ is not $g$-degenerate.
\begin{prop} \label{prop:smalle}
If $b \leq 1$ and $Y \subseteq L'_{a,b}$ is $\ell$-degenerate, then $Y \subseteq Y(G, e)$ for some $(G, e) \in \cG'$ (where $G$ has $a+b+2$ edges and $b+2$ vertices).
\end{prop}
\begin{proof}
The case $b=0$ is trivial, as $L'_{a,0} = Y(G, e)$, where $G$ is the series-parallel graph with two vertices and $a+2$ parallel edges, and $e$ is an arbitrary edge of $G$.

Consider $Y \subseteq L'_{a,1}$ for some $a \geq 0$. For some $a' \geq a$ and $b' \geq 1$, define a random subset $S \subseteq L'_{a',b'}$ as follows: randomly choose vectors $v_0,v_1,\ldots,v_{a'+b'+1} \in \setf_2^{b'+1}$, and include a string $s\in L_{a',b'}$ in $S$ if and only if $v_0,v_{i_1},\ldots,v_{i_{b'}}$ and $v_{i_*},v_{i_1},\ldots,v_{i_{b'}}$ both form bases of $\setf_2^{b'+1}$, where $i_1,\ldots,i_{b'}$ are the locations of the $1$s and $i_*$ is the location of of the $\ast$. The probability that an arbitrary $s$ is in $S$ is the probability that $v_{i_1},\ldots,v_{i_{b'}}$ are linearly independent times $1/4$, which is $\frac{1}{4}\prod_{i=1}^{b'} (1-2^{-i-1}) > \frac{1}{2}\prod_{i=1}^\infty (1-2^{-i}) > 0.1443$. Thus the expected value of $\abs{S}$ is greater than $0.1443 \abs{L'_{a',b'}}$, so by $\ell$-degeneracy there exists some $a'\geq a$, $b' \geq 2$, choice of $v_0,v_1,\ldots,v_{a'+b'+1}$, and choice of $p \in P'_{a,1,a',b'}$ such that $\psi_p(Y) \subseteq S$.

Suppose the $1$s in $p$ occur in positions $i_1,\ldots,i_{b'-1}$ and $s_k$ occurs at position $j_k$. If $v_0,v_{i_1},\ldots,v_{i_{b'-1}}$ are linearly dependent, then the image of $\psi_p$ is disjoint from $S$, a contradiction. Thus $v_{i_1},\ldots,v_{i_{b'-1}}$ span a subspace $W \subseteq \setf_2^{b'}$ of codimension $2$ not containing $v_0$. For every $k \in [a+2]$ let $w_k$ be the image of $v_{j_k}$ in the $2$-dimensional space $\setf_2^{b'+1}/W$. This induces a partition $[a+2] = I_0 \sqcup I_1 \sqcup I_2 \sqcup I_3$, where $I_0$ consists of the $k$ for which $w_k = 0$, $I_1$ consists of the $k$ for which $w_k$ is the image of $v_0$, and $I_2,I_3$ consist of the $k$ for which $w_k$ is equal to each of the other two nonzero vectors in $\setf_2^{b'}/W$. If $s \in X$ has a $1$ in position $k_1$ and a $\ast$ in position $k_*$, then $w_{k_1}$ and $w_{k_*}$ must be nonzero and distinct, and additionally we cannot have $k_1 \in I_1$. It follows that $s \in Y(G, e)$, where $(G, e)$ is the multigraph with distinguished edge shown in \cref{fig:smalle}.
\end{proof}

By duality, \cref{prop:smallv,prop:smalle} also classify all $\ell$-degenerate subsets of $L_{a,b}$ for $a \leq 2$ and $L'_{a,b}$ for $a \leq 1$. Thus, $L_{3,3}$ and $L'_{2,2}$ are the smallest vertex- and edge-layers for which $\ell$-degeneracy is not completely understood. Perhaps the most intriguing special case is given by the following two subsets.
\begin{ques} \label{ques:k4}
Let
\[X_{16} = L_{3,3} \setminus \set{010101,011010,100110,101001}\]
be a set of size $16$ and $Y_{18} \subseteq L'_{2,2}$ be the set of $18$ edges with both endpoints in
\[(L_{2,3} \cup L_{3,2}) \setminus \set{00011,01100,10101,11010}.\]
Are $X_{16}$ and $Y_{18}$ $\ell$-degenerate? 
\end{ques}
The patterns $X_{16}$ and $Y_{18}$ are noteworthy for two reasons. For one, they are the largest subsets of $L_{3,3}$ and $L'_{2,2}$, respectively, for which $\ell$-degeneracy is not ruled out by the $\setf_2$-based construction used in the proofs of \cref{prop:smallv,prop:smalle}. Moreover, they are equal to $X(K_4)$ and $Y(K_4, e)$ ($e \in E(K_4)$ is arbitrary), respectively, if the definitions of $X(G)$ and $Y(G,e)$ are extended in the natural way to include $G$ that are not series-parallel. Since $K_4$ is the simplest multigraph that is not series-parallel, \cref{ques:k4} in some sense asks whether the series-parallel condition in \cref{thm:mainv,thm:maine} is necessary.

\section{Concluding Remarks} \label{sec:conc}
\subsection{Different notions of degeneracy} \label{sec:containment}
This paper contains results relating to several notions of degeneracy (most notably $g$-, $h$-, and $\ell$-degeneracy) stemming from different notions of pattern containment. Here, we will briefly explain why these notions fundamentally describe the same phenomenon.

Consider a connected subgraph $H$ of the hypercube. Clearly, if some embedding of $H$ in the hypercube is an $h$-degenerate edge pattern, then $H$ is $g$-degenerate. Moreover, the converse is also true. Indeed, there are finitely many ways to realize $H$ as an edge pattern; call them $Y_1, \ldots, Y_m$. If none of $Y_1, \ldots, Y_m$ are $h$-degenerate, there must exist $\delta_1,\ldots,\delta_m > 0$ such that for every $n$ there exists a subset $S_i \subseteq E(Q_n)$ of size at least $\delta_i e(Q_n)$ that doesn't contain $Y_i$. (Note that by a standard averaging argument, $\ex(Q_n, Y_i)/e(Q_n)$ is a decreasing function of $n$.) Observe that $Q_n$ is edge-transitive, so given an arbitrary edge $e \in E(Q_n)$, the probability that $e \in \sigma(S_i)$ for a uniformly random automorphism $\sigma$ is $\abs{S_i}/e(Q_n) \geq \delta_i$. Thus, the expected size of $S=\bigcap_{i=1}^m \sigma_i(S_i)$ for uniformly random automorphisms $\sigma_1,\ldots,\sigma_m$ is at least $\delta_1\delta_2 \cdots \delta_m e(Q_n)$. Since $S$ doesn't contain any $Y_i$ and hence cannot contain $G$, we have shown that $\ex(Q_n, H) \geq \delta_1\delta_2 \cdots \delta_m e(Q_n)$ and thus that $H$ is not $g$-degenerate.

Similar arguments show the following results:
\begin{itemize}
\item The quantity $L(d)$, defined in \cref{sec:quant} to be the size of the largest $h$-degenerate subset of $V(Q_d)$, is also the largest $m$ such that the size of the largest subset of $V(Q_n)$ intersecting every $d$-dimensional subcube in fewer than $m$ points is $o(2^n)$. This was the original definition given in \cite{JT10}.
\item A vertex or edge pattern is $h$-degenerate if and only if every way to embed it in a layer is $\ell$-degenerate.
\end{itemize}

\subsection{Completeness} \label{sec:opt}
An obvious unanswered question is whether the $h$-degenerate patterns found in this paper are all the $h$-degenerate patterns that exist. We believe there are reasons to take either side, which we will briefly discuss below.

The main reason to believe that there are no more $h$-degenerate patterns is by analogy with other Tur\'an-type problems. As mentioned in the introduction, the degenerate $r$-uniform hypergraphs are the $r$-partite $r$-uniform hypergraphs, which are precisely the hypergraphs that can be obtained by blowing up a single edge and possibly deleting some edges. Here, a similar blowing-up process appears in the form of duplication and coduplication, with the structure of series-parallel graphs arising naturally from the interactions of these two operations. It would be natural to conjecture that the $h$-degenerate patterns are precisely those obtained from duplicating and coduplicating trivial patterns.

On the other hand, to show that no other $h$-degenerate patterns exist would require constructing dense subsets of the hypercube that avoid certain patterns. Currently, the only known method for accomplishing this is to use linear algebra over $\setf_2$, which exploits a rigid algebraic structure. There is no particular reason to believe that algebraic constructions ruling out more patterns must exist.

\subsection{Properties of degenerate patterns}
The sets $X(G)$ and $Y(G, e)$ have a number of interesting properties, and it may be possible to show that some of these properties hold for all degenerate patterns without classifying them. For instance, $Y(G, e)$ is defined to be an induced subgraph of the hypercube.
\begin{ques}
Is every maximal $h$- (or $\ell$-)degenerate edge pattern an induced subgraph of the hypercube?
\end{ques}

Moreover, the sets $X(G)$ and $Y(G,e)$ can be related to each other. Given $Y\subseteq L'_{a,b}$, define the set $\Phi(Y) \subseteq L_{a+1,b+1}$ to consist of strings such that deleting the last character yields an endpoint of an edge in $Y$. Note that if $(G, e) \in \cG'$, we have $\Phi(Y(G, e)) = X(G)$. Conversely, given $X \in L_{a+1,b+1}$ and $i \in [a+b+2]$, define $\Psi_i(X)\subseteq L'_{a,b}$ to be edges of the graph induced on the image of $X$ under the map $L_{a+1,b+1} \to L_{a+1,b} \cup L_{a,b+1}$ that forgets the $i$ coordinate. Note that $(G, e) \in \cG'$ and $e$ is the $i$th edge of $G$, then $\Psi_i(X(G)) = Y(G, e)$.
\begin{ques}
Suppose $X \in L_{a+1,b+1}$ is $\ell$-degenerate. Must $\Psi_i(X)$ be $\ell$-degenerate for all $i \in [a+b+2]$?
\end{ques}
\begin{ques}
Suppose $Y \in L'_{a,b}$ is $\ell$-degenerate. Must $\Phi(Y)$ be $\ell$-degenerate?
\end{ques}

\subsection{Relation to matroids}
The vertex patterns that appear in this paper are often the basis sets of matroids. For instance, for $G \in \cG$, the set $X(G)$ is the basis set of the \vocab{series-parallel matroid} $M(G)$; more generally, duplication and coduplication correspond to \vocab{parallel} and \vocab{series extensions} of matroids, respectively. Moreover, the construction in the proof of \cref{prop:smallv} is simply the basis set of a randomly generated binary matroid. It would be interesting to explore this connection further; an especially tantalizing possibility would be if there exists a well-known family of matroids whose basis sets yield new constructions ruling out the degeneracy of certain vertex patterns.

One possible direction is the following: recent work of van der Pol, Walsh, and Wigal \cite{Pol2025} defined, for nonnegative integers $n \geq r$ and a uniform matroid $U$, the quantity $\ex_{\sM}(n, r, U)$ to be the maximum number of bases an $n$-element, rank-$r$ matroid can have while avoiding $U$ as a minor. It is possible to extend this definition to all finite matroids, by calling a matroid $M$ a \vocab{weak minor} of $N$ if a relaxation of $M$ is a minor of $N$. Then, we let $\ex_\sM(n, r, M)$ be the maximum number of bases an $n$-element, rank-$r$ matroid can have while avoiding $M$ as a weak minor. If $X$ and $X'$ are the basis sets of $M$ and $N$, interpreted as subsets of the hypercube, $M$ is a weak minor of $N$ if and only if there is some oriented embedding $\psi$ such that $\psi(X)\subseteq X'$. As a result, we find that $\ex_\sM(n, r, M) \leq \ex(L_{n-r,r}, X)$.

Given this definition of $\ex_\sM(n, r, M)$, we may now define the degeneracy of matroids similarly to $\ell$-degeneracy; the above inequality shows that if the basis set of a matroid is $\ell$-degenerate, the matroid is degenerate. In particular, the results of this paper show that series-parallel matroids are degenerate. Moreover, by the same argument as \cref{prop:smallv}, every degenerate matroid must have a restriction that is a binary matroid. An interesting question, parallel to \cref{ques:k4}, is whether the graphical matroid of $K_4$ is degenerate.

\section*{Acknowledgments}
This work was supported by the NSF Graduate Research Fellowships Program (grant number: DGE-2039656). The author would like to thank Noga Alon and Maria Axenovich for helpful comments and discussions.

\printbibliography
\end{document}